\def\datum{ }
\documentclass[12pt]{amsart}
\usepackage{amscd,amsmath,amssymb,amsfonts,ifthen}

\usepackage{bbold}
\usepackage{graphicx}
\usepackage{url,mathrsfs,euscript}
\usepackage[all]{xypic}

\usepackage{charter,hyperref}
\usepackage{eulervm}

\theoremstyle{plain}
\newtheorem{thm}{Theorem}
\newtheorem{lem}[thm]{Lemma}
\newtheorem{cor}[thm]{Corollary}
\newtheorem{prop}[thm]{Proposition}

\newtheorem{ex}[thm]{Example}

\theoremstyle{definition}
\newtheorem{rmk}[thm]{Remark}

\numberwithin{thm}{section}

\def\mod{{\rm mod}}
\def\comod{{\rm comod}}
 
\date{\datum}

\author{Ph\`ung H\^o Hai}\title{On an injectivity lemma in the proof of Tannakian duality}
\email{phung@math.ac.vn}
\address{Institute of Mathematics, Vietnam Academy of Science and Technology, 18 Hoang Quoc Viet, Cau Giay, Hanoi, Vietnam}
\thanks{This research is funded by Vietnam National Foundation for  Science and Technology Development(NAFOSTED) under grant number 101.01-2011.34. Part of this work has been carried out when the author was visiting the Vietnam Institute for Advanced Study in Mathematics.} 
\subjclass[2010]{16T15, 18A22}
\date{May 12, 2015}

\begin{document}
\begin{abstract}
In this short work we give a very short and 
elementary proof of the injectivity lemma, 
which plays an important role in the Tannakian duality for Hopf 
algebras over a field. Based on this we provide some generalizations of this fact to the case of flat algebras over a noetherian domain. 
\end{abstract} 
\maketitle  
\parskip10pt
\section*{Introduction}
Let $k$ be a field.
The neutral Tannakian duality establishes a 
dictionary between $k$-linear tensor abelian 
categories, equipped with a fiber functor to the 
category of $k$-vector spaces, and affine group 
schemes over $k$. The duality was first obtained 
by Saavedra in \cite{SR72},
among other important results. In \cite{DM}, 
Deligne and Milne gave a very readable 
self-contained account on the result. 

The main part of the proof of Tannakian duality 
was to establish the duality between abelian 
category equipped with fiber functors to ${\rm 
vect}_k$ and $k$-coalgebras. Here, one first 
proves the claim for those categories which have 
a (pseudo-) generator. Such categories are in correspondence to finite 
dimensional coalgebras. 
The injectivity lemma 
claims that, under this correspondence, fully 
faithful exact functors, which preserves 
subobjects, correspond to injective 
homomorphisms of coalgebras (see Lemma 
\ref{lem2} for the precise formulation). This lemma was implicitly used in the proof of Prop.~2.21 in \cite{DM}. 
In the original work of Saavedra  this claim was 
obtained as a corollary of the duality, cf. 
\cite[2.6.3~(f)]{SR72}. In his recent book 
Szamuely gave a more direct proof of the 
injectivity lemma, cf. \cite[Prop.~6.4.4]{Sza}. 
Szamuely's proof is nice but still quite 
involved. Similar treatment and some generalizations was also made in Hashimoto's book \cite[Lem.~3.6.10]{Ha}.

In this short work we give a very short and 
elementary proof of the injectivity lemma. Based on this, we provide some 
generalizations of this fact to the case of flat 
coalgebras over a noetherian domain. We believe that these results
will find applications in the Tannakian duality for Hopf algebras
over noetherian domains.

\subsection*{Notations}Throughout this work we shall fix a commutative noetherian domain $R$. The quotient field of $R$ will be denoted by $K$. 

For an algebra (or more general a ring) $A$ (commutative or not), ${\rm mod}(A)$ denotes the category of left $A$-modules and $\mod_{\rm f}(A)$ denotes the subcategory finitely generated modules.

For a coalgebra $C$ over  $R$, ${\rm comod}(C)$ denotes the category of right $C$-comodules, and $\comod_{\rm f}(C)$ denotes the subcategory of comodules which are {\em finite over $R$}.

\section{A simple proof of the injectivity lemma} \label{sect1}
Let $k$ be a field and $f:A\to B$ be a 
homomorphism of finite dimensional $k$-algebras. 
Then $f$ induces a functor $\omega:
\mod_{\rm f}(B)\to\mod_{\rm f}(A)$ between the categories of 
finite modules over $A$ and $B$, which is 
identity on the underlying vector spaces. In 
particular, $\omega$ is a faithfully exact functor.

\begin{lem}\label{lem1} A homomorphism $f$ 
given as above is surjective if and only if the 
induced functor $\omega$ is full and has the 
property: for any $B$-module $X$ and any $A$-submodule $Y$ of $\omega(X)$, there exists a 
$B$-submodule $X'$ of $X$ such that 
$\omega(X')=Y$. In other words $f$ is surjective 
if and only if $\mod_{\rm f}(B)$, by means of $f$, is a 
full (abelian) subcategory of $\mod_{\rm f}(A)$, closed under 
taking submodules.\end{lem}

\begin{proof}

If $f$ is surjective then obviously $\omega$ has 
the claimed properties. We prove the converse 
statement. Thus for any $B$-module $X$ and 
any submodule $Y$ of $X$, considered as 
modules over $A$, we know that $Y$ is also 
stable under the action of $B$ (obtained by 
restricting the action of $B$ on $X$). Assume the 
contrary that $f$ is not surjective, i.e.,
$B_0:={\rm im} (f)$ is a strict subalgebra of 
$B$. Then $B_0\subset B$ is an inclusion of $A$-modules, $B$ it self is a $B$-module, but $B_0$ is not stable under the action of $B$ as it contains the unit of $B$. A contradiction.
\end{proof}

By duality we have the following result for 
comodules.

\begin{lem}\label{lem2}
Let $f:C\to D$ be a homomorphism of finite 
dimensional $k$-coalgebras. Then the category 
$\comod_{\rm f}(C)$, considered by means of $f$ as a 
subcategory of $\comod_{\rm f}(D)$, is full and closed 
under taking subobjects if and only if $f$ is 
injective.
\end{lem}

\begin{rmk}\label{rmk3}

In the proof of Lemma \ref{lem1}, there is no need to 
assume that $A$ is finite dimensional. Therefore, 
in Lemma \ref{lem2} there is no need to assume that 
$D$ is finite dimensional. On the other hand, it is 
known that each coalgebra is the union of its 
finite dimensional subcoalgebras. Therefore there is no need to impose the dimension condition on 
$C$ either.

\end{rmk}

\section{Generalizations}

We give here several generalizations of the 
lemmas in Section \ref{sect1} to the case when 
$k$ is a noetherian domain.

  \subsection{The case of algebras}     
Let $R$ be a noetherian domain. We consider $R$-algebras. Modules over such an algebra are automatically $R$-modules, we call such a module $R$-finite (resp. torsion-free, flat, projective, free) if it is finite (resp. torsion-free, flat, projective, free) over $R$. Let $R\to S$ be a homomorphism of commutative algebras. Then the base change $R\to S$ will be denoted by the subscript $( )_S$. For instance $M_S:=M\otimes_RS$, $f_S:=f\otimes_RS$ for an $R$-linear map $f$.

Let  $f:A\to B$ be a 
homomorphism of finite, torsion-free $R$-algebras. It 
induces a functor $\omega:\mod_\text{f}(B)\to 
\mod_\text{f}(A)$, which is 
identity functor on the underlying $R$-modules, 
therefore it is faithful and exact. The following lemma is a straightforward generalization of Lemma \ref{lem1}.
\begin{lem}\label{lem4} The map $f$ as above is surjective if and only if 
$\mod_{\rm f}(B)$ when considered by means of $f$
  as a subcategory of $\mod_{\rm f}(A)$ is full and closed 
under taking subobjects.
\end{lem}

\begin{proof} Only the ``if'' claim needs 
verification. Let $\mathfrak m\in R$ be a 
maximal ideal and let $k(\mathfrak 
m):=R/\mathfrak m$ be the residue field. Then 
the full subcategory of $\mod_{\rm f}(A)$ annihilated by 
$\mathfrak m$ is equivalent to 
$\mod_{\rm f}(A_{k(\mathfrak m)})$. Note that this subcategory is also closed under taking subobjects.

Thus, by assumption, 
$\mod_{\rm f}(B_{k(\mathfrak m)})$ is a full 
subcategory of $\mod_{\rm f}(A_{k(\mathfrak m)})$,
closed under taking subobjects. Therefore the 
map $f_{k(\mathfrak m)} :A_{k(\mathfrak m)}\to 
B_{k(\mathfrak m)}$ is surjective, by means of 
Lemma \ref{lem1}. This holds for any maximal ideal 
$\mathfrak m$ of $R$, hence
$(B/f(A))_{k(\mathfrak m)}=0$ for all 
maximal ideals $\mathfrak m$. According to 
\cite[Thm~4.8]{Mats}, we conclude that 
$B/f(A)$ itself is zero.
\end{proof}

Lemma \ref{lem4} can also be proved in the same way as Lemma \ref{lem1}.
We prefer to present the above proof as its main idea (using base change) will be exploited further. 
 
An $A$-submodule $N$ of $M$ is said to be 
{\em saturated} iff $M/N$ is $R$-torsion-free.
A homomorphism $f:A\to B$ is said to be
{\em dominant} if 
$f_K :A_K \to B_K$ is surjective, or 
equivalently $B/f(A)$ is $R$-torsion. Here $K$ denotes the quotient field of $R$.

\begin{prop}\label{prop5}
Let $f:A\to B$ be a homomorphism of $R$-torsion 
free algebras and assume that $B$ is 
$R$-finite. Let $\omega:\mod_{\rm f}(B)\to\mod_{\rm f}(A)$ be the induced functor. Then:
\begin{enumerate}\item The image of 
$\omega$ is closed under taking saturated 
subobjects of $R$-torsion-free objects iff $f$ 
is dominant. In this case $\omega$ is also closed under taking saturated submodules of any modules and its restriction to the subcategory of $R$-torsion-free modules is full.
\item The image of $\omega$ is closed 
under taking subobjects of $R$-torsion-free 
objects iff $f$ is surjective. (In this case 
$\omega$ is also obviously full.)
\end{enumerate}

\end{prop}

\begin{proof}

(1). Assume that $\omega$ has the required 
property. We show that $f$ is dominant, i.e. 
$f_K$ is surjective. It suffices to show that the 
functor $\omega_K:\mod(B_K)\to \mod(A_K)$ 
induced from $f_K$ satisfies the condition of 
Lemma \ref{lem1}. Let $X$ be a finite $B_K$-module and $Y\subset X$ an $A_K$-submodule. 
Consider $X$ as a $B$-module, take a $K$-basis of $X$ such that a part of it is a basis of $Y$ and let $M$ be the $B$-submodule generated by this basis, let $N:=M\cap Y$. Then $N_K=Y$ (as it contains a $K$-basis of $Y$). By the diagram 
below $M/N$ is $R$-torsion-free.
$$\xymatrix{0\ar[r]&N\ar[r]\ar@{^(->}[d]& 
M\ar[r]\ar@{^(->}[d]& M/N\ar[r]\ar@{^(->}
[d]&0\\
0\ar[r]& Y\ar[r]& X\ar[r]&X/Y\ar[r]&0.}$$
Thus $N$ is a saturated submodule of $M$, hence, by assumption, $N$ is stable under $B$, consequently
$Y=N_K$ is stable under $B_K$.

Conversely, assume that $f$ is dominant, i.e., $f_K$ is surjective. Then 
$\omega_K$ is fully faithful and closed under 
taking submodules. Let $M$ be an $R$-torsion-free 
$B$-module, $N\subset M$ be a saturated 
$A$-submodule. Then $M/N$ is also 
$R$-torsion-free, hence $N=M\cap N_K$. Now 
$N_K$ is stable under $B_K$ and $M$ is 
stable under $B$, showing that $N$ is stable 
under $B$.

 Let 
$\varphi:M\to N$ be an $A$-linear map, where 
$M,N$ are 
both $R$-torsion-free then $\varphi$ is 
determined by $\varphi_K:M_K\to N_K$. Since 
$f_K$ is surjective, we know that $\varphi_K$ is 
$A_K$-linear, hence also $B_K$-linear, implying that $f$ is $B$-linear. Thus $\omega$ restricted to $R$-torsion free modules is full.

Let now $M$ be a finite $A$-module, $N$ be a finite $B$-module and $\varphi:M\to N$ be an injective $A$-linear map with $M/\varphi(N)$ being $R$-torsion free. Consider a finite free 
$B$-linear cover $\psi:N'\to N$. Let 
$\varphi':M'\to N'$ be the pull-back of 
$\varphi$ along $\psi$ (as $A$-modules). 
$$\xymatrix{M'\ar@{->>}[d]_{\psi'}\ar@{^(->}[r]^{\varphi'}&N'\ar@{->>}[d]^\psi\\ M\ar@{^(->}[r]_\varphi &N.}$$
Then $\varphi'$ is 
injective and $\psi':M'\to M$ is surjective, moreover $N'/\varphi'(M')\cong N/\varphi(N)$ hence is $R$-torsion free.
Consequently, $M'$ is $B$-stable and, since $\varphi\psi'=\psi\varphi'$ is $B$-linear, there is a $B$-action on $M$ making $\varphi$ $B$-linear.

(2). According to the proof of Lemma \ref{lem4}, 
it suffices to show that for any maximal ideal 
$\mathfrak m$ of $R$, the image of 
$\omega_{k(\mathfrak m)}$ is closed under 
taking submodules. Let $V$ be a 
$B_{k(\mathfrak m)}$-module and let 
$\varphi:U\to V$ be an inclusion of 
$A_{k(\mathfrak m)}$-modules. Represent $V$ 
as a quotient of some (free) $B$-modules $M$, 
then $U$ will be a quotient of some $A$-submodule $N$ of $M$. We have $N$ is $R$-torsion-free, hence, by 
assumption, $N$ is stable under $B$,  so that
$U$ must be $B_{k(\mathfrak m)}$-stable.
\end{proof}

\subsection{The case of coalgebras}
In this subsection we consider $R$-flat 
coalgebras. For such coalgebras the comodule 
categories are abelian, see, e.g. \cite[Lem.~3.3.3]{Ha} or \cite[Thm.~3.13]{BW}. Moreover, the kernel and image of a comodule  homomorphism are the same as those of the underlying $R$-module homomorphism. A homomorphism of flat $R$-coalgebras $f:C\to D$ induces the restriction functor $\omega:\comod(C)\to \comod(D)$ which is the identity functor on the underlying $R$-module. Hence $\omega$ is faithful and exact.

We say that a homomorphism of 
flat $R$-coalgebras $f:C\to D$ is pure if it is  a pure homomorphism of 
$R$-modules. This 
condition is the same as requiring $D/f(C)$ be $R$-flat.  Note 
also that over a noetherian domain, finite flat 
modules are projective.

For the case $C$ and $D$ are $R$-projective and $C$ is $R$-finite, the desired results can be deduced from the previous subsection by means of the following lemma.
\begin{lem}\label{lem6} Let $C$ be an $R$-finite 
flat (hence projective) module and $D$ be an 
$R$-projective module. Then:
\begin{enumerate}\item $f$ is injective 
iff $f^\vee:D^\vee\to C^\vee$ is dominant;
\item $f$ is injective and pure iff 
$f^\vee:D^\vee\to C^\vee$ is surjective, 
\end{enumerate}
where $C^\vee:={\rm Hom}_R(C,R)$.
\end{lem}
\begin{proof}
Embedding $D$ as a direct summand into a free 
module does not 
change the properties of $f$ and $f^\vee$, 
hence we can assume that $D$ is free. Since $C$ 
is finite, there exists a finite direct summand of 
$D$ which contains the image of $f$ and we can 
replace $D$ by this summand, that means 
we can assume that $D$ is finite. 
The claims for finite projective modules are obvious. 
For (1), it involves only the generic fibers. For (2),
$f:C\to D$ is pure iff $D/f(C)$ projective,  and iff the 
sequence $0\to C\to D\to D/f(C)\to 0$ is split exact, iff 
the sequence 
$0\to (D/f(C))^\vee\to D^\vee\to C^\vee\to 0$ is split
exact. 
\end{proof}


A subcomodule $N$ of a $C$-comodule $M$ said to be saturated if $M/N$ is $R$-torsion free.
\begin{prop}\label{prop9}
Let $C,D$ be $R$-projective coalgebras. Let 
$f:C\to D$ be a homomorphism of $R$-coalgebras and $\omega:\comod_{\rm f}(C)\to\comod_{\rm f}(D)$ be the induced functor. 
Assume that $C$ is $R$-finite.
Then:
\begin{enumerate}
\item The image of functor $\omega$ is closed 
under taking saturated subcomodules of $R$-
torsion-free comodules iff $f$ is injective. In this case $\omega$ is also closed under taking saturated subcomodules of any comodules and its restriction to the subcategory of $R$-torsion-free comodules is full.

\item  
The image of functor $\omega$ is closed 
under taking subcomodules of $R$-torsion-free 
comodules iff $f$ is injective and pure. (In 
this case $\omega$ is also obviously full.)
\end{enumerate}
\end{prop}

\begin{proof}
 Since $D$ is projective, the natural functor 
${\rm comod}(D)\to \mod(D^\vee)$ is fully faithful, exact with image closed under taking subobjects \cite[3.10]{Ha}.
Thus the functor, induced from $\omega$, $\mod(D^\vee)\to \mod(C^\vee)$ is fully faithful, exact and has image closed under taking subobjects iff $\omega$ is. The claim follows from Proposition \ref{prop5} and Lemma \ref{lem6}.
\end{proof}

We say that a flat $R$-coalgebra $C$ is locally finite if 
$C$ is the union of its finite $R$-projective pure
subcoalgebras $C_\alpha$, $\alpha\in A$. 
This property is called IFP (ind-finite projective property) in \cite{Ha}. 
As a corollary of Proposition \ref{prop9}, we have
\begin{cor} Let $C,D$ be projective $R$-coalgebras. 
Assume that $C$ has IFP, $C=\bigcup C_\alpha$. 
Let $f:C\to D$ be a homomorphism of $R$-coalgebras. 
Then $f$ is injective iff the induced functor $\omega:\comod_{\rm f}(C)\to\comod_f(D)$ has image closed under taking subobjects. 
In particular, $\comod_{\rm f}(C)$ is the union of its full 
subcategories $\comod_{\rm f}(C_\alpha)$, which are closed under taking 
subobjects.\end{cor}

Notice that 
 there exist $R$-flat coalgebras which contains almost no finite pure subcoalgebras, as shown in the examples below.

\begin{ex}[{\cite{dSantos}}]\label{ex2.5}
\rm
  (1) Assume that $R$ is a Dedekind ring with characteristic equal to 2 and $\pi\neq 0$ is a non-unit in $R$. Consider the algebra $H:=R[T]/(\pi T^2+T)$ with the coalgebra structure given by
$$\Delta(T)=T\otimes 1+1\otimes T.$$  
Then $H$ is torsion-free, hence flat over $R$.
Let $M\subset H$ be a saturated $R$-finite subcomodule of $H$. Then $M$ is free over $R$, more over its rank is at most 2, as $H_K$ has dimension 2 over $K$. If $M$ has rank 2 then $H/M$ is $R$-torsion, contradiction, hence $M$ has rank 1 over $R$. The coaction of $H$ on $M$ is thus given by a group-like element in $H$. On the other hand, the $R$-flatness of $H$ implies $H\otimes H$ is a submodule of $H_K\otimes_KH_K$, hence the coproduct on $H$ is the restriction of the coproduct on $H_K$:
$$\xymatrix{H\ar[r]^{\Delta\quad}\ar[d]& H\otimes H\ar[d]\\
H_K\ar[r]_{\Delta\qquad}& H_K\otimes_KH_K.}$$
Consequently, a group-like element of $H$ is a group-like element of $H_K$. But in $H_K$ the unique group-like element is 1. Thus $M$ is a trivial comodule of $H$.
  
(2) Similarly, assume that $R$ is a Dedekind ring, in which 2 is invertible. Consider the algebra $H:=R[T]/(\pi T^2+2T)$ where $\pi\neq 0$ is a non-unit. $H$ is a Hopf algebra with the coaction given by 
$$\Delta(T)=T\otimes 1+1\otimes T+\pi T\otimes T.$$
Similar discussion shows that there are only two finite saturated subcomodules of $H$, the one generated by 1 and the other generated by the group-like element $1+\pi T$.
\end{ex}

To treat a general coalgebra homomorphism $f:C\to D$ we shall imitate the proof of Lemma \ref{lem1}. Our condition on $\omega$ will be some what stronger.

\begin{prop}\label{prop9b}
Let $C,D$ be $R$-flat coalgebras  and  
$f:C\to D$ be a homomorphism of $R$-coalgebras. 
Let $\omega:\comod(C)\to\comod(D)$ be the induced functor on comodule categories.
Then:
\begin{enumerate}
\item The image of functor $\omega$ is closed 
under taking saturated subcomodules of $R$-torsion-free comodules iff $f$ is injective. In this case $\omega$ is also closed under taking saturated subcomodules of any comodules and its restriction to the subcategory of $R$-torsion-free comodules is full.

\item 
The image of functor $\omega$ is closed 
under taking subcomodules of any 
comodules iff $f$ is injective and pure. In 
this case $\omega$ is full.
\end{enumerate}
\end{prop}
\begin{proof}
(1) Assume that $\omega$ has the required property. Let $C_0:={\rm ker}(f)$. 
Then $C/C_0\cong {\rm im}(f)\subset D$ is $R$-torsion-free, hence $C_0$ is a saturated subcomodule of $C$, considered as $D$-comodules (the big left square with curved vertical arrows below is commutative). 
By assumption, the coaction of $D$ on $C_0$ lifts to a coaction of $C$. That is, there exists a coaction $C_0\to C_0\otimes C$ (the dotted arrow below) making the following diagram commutative:
$$\xymatrix{
C_0\ar[r] \ar@{-->}[d]\ar@/_30pt/[dd]
&C\ar[d]^\Delta\ar[r]^f\ar@/^30pt/[dd]
& D\ar[dd]^\Delta \\
C_0\otimes C\ar[r]\ar[d]^{{\rm id}\otimes f}& C\otimes C\ar[d]^{{\rm id}\otimes f}&\\
C_0\otimes D\ar[r]&C\otimes D\ar[r]_{f\otimes {\rm id}}&D\otimes D.}$$
In particular, the coaction of $C$ on $C_0$ is the restriction of that on $C$ (the upper-left square).
That is, for any $c\in C_0$ we have a representation
$$\Delta(c)=\sum_{(c)}c_{(1)}\otimes c_{(2)},$$
with $c_{(1)}\in C_0$. On the other hand, as $f$ is a coalgebra homomorphism, we have $\varepsilon\circ f=\varepsilon$. Consequently, $\varepsilon(C_0)=0$. Applying $\varepsilon\otimes{\rm id}$ to the above equation we get
$$c=\sum_{(c)}\varepsilon(c_{(1)})\otimes c_{(2)}=0.$$
A contradiction. Thus ${\rm ker}(f)=0$.

Conversely, assume that $f:C\to D$ is injective, then the map $f_K:C_K\to D_K$ is also injective, as the base change $R\to K$ is flat. Hence, according to \ref{lem2}, \ref{rmk3}, $\omega_K$ is fully faithful and is closed under taking subcomodules. Thus $\omega$ is full when restricted to $R$-torsion free comodules. 

Finally we show that the image of $\omega$ is closed under taking saturated subcomodules of any comodules. For an $R$-module $M$, let $M^\text{tor}$ denote its torsion part, i.e. those elements of $M$ killed by some non-zero element of $R$. Then we have
$$M^\text{tor}={\rm ker}(M\to M\otimes K).$$
Therefore, for any flat $R$-module $P$ we have
$$(M^\text{tor}\otimes P)\cong (M\otimes P)^\text{tor}.$$
Since any $R$-linear map preserves the torsion part, we conclude that, if $M$ is a $C$-comodule then $M^\text{tor}$ is a subcomodule. 
Let now $N\subset M$ be a subcomodule with respect to the action of $D$. 
If $M/N$ is $R$-torsion free then $M^\text{tor}=N^\text{tor}$ and hence is stable under the coaction of $C$. 
Hence we can consider the saturated inclusion $N/N^\text{tor}\subset M/M^\text{tor}$, which by assumption shows that $N/N^\text{tor}$ is stable under the coaction of $C$. 
As $C$ is flat, we conclude that $N$ itself is stable under $C$.

(2) Assume that $\omega$ is closed under taking subcomodules of any comodules. Then according to (1), $f$ is injective. Assume $f$ is not pure, then there exists an ideal $I$ of $R$ such that the induced map $R/I\otimes C\to R/I\otimes D$ is not injective. 
Let $C_0$ be the kernel of this map. 
Repeat the argument of the proof of (1) we conclude that $C_0$ is stable under the coaction of $D$ but not under the coaction of $C$, a contradiction. Thus $f$ has to be pure.

For the converse, assuming that $f:C\to D$ is injective and pure and $N\subset M$ be $R$-modules, then we have the equality of submodules of $M\otimes D$:
$$N\otimes D\cap M\otimes C=N\otimes C,$$
where $C$ is considered as an $R$-submodule of $D$ by means of $f$. Hence, if $M$ is a $C$-comodule and $N$ is a $D$-subcomodule of $M$, then, denote by $\delta$ the coaction, we have
$$\delta(N)\subset N\otimes D\cap M\otimes C=N\otimes C.$$
That is, $N$ is stable under the coaction of $C$.
\end{proof}

\begin{rmk}
According to Serre \cite[Prop.~2]{Se}, any object in $\comod(C)$ is the union of its $R$-finite subcomodules (but generally not saturated). I don't know if one can prove Propsition \ref{prop9b} with $\comod(C)$, $\comod(D)$ replaced by $\comod_f(C)$, $\comod_f(D)$, respectively.
\end{rmk}

\section*{Acknowledgment}
The author would like to thank Dr. Nguyen Chu Gia Vuong and Nguyen Dai Duong for stimulating discussions and VIASM for providing a very nice working space and the financial support.
A special thank goes to the anonymous referee for his careful and throughout reading. His comments and remarks have greatly improved the work.

\end{document}